\documentclass[amsfonts]{amsart}

\usepackage{mathpazo}
\usepackage{textcomp}
\usepackage{nicefrac} 
\usepackage{xfrac}    
\usepackage{frcursive}
\usepackage[T1]{fontenc}

\usepackage{amsmath}
\usepackage{amsfonts}
\usepackage{amssymb}
\usepackage{makerobust}
\MakeRobustCommand\overrightarrow
 
\usepackage{mathtools}

\usepackage[all,arc]{xy}
\usepackage{enumerate}
\usepackage{mathrsfs}
\usepackage[utf8]{inputenc}
\usepackage{soul}
\usepackage{hyperref}

\DeclareMathOperator{\Ima}{Im}

\DeclareMathOperator{\dv}{div}

\newcommand{\Rr}{\mathbb{R}}

\newcommand{\Cc}{\mathbb{C}}

\newtheorem{thm}{Theorem}[section]

\newtheorem{prop}[thm]{Proposition}
\newtheorem{defi}[thm]{Definition}
\newtheorem{rema}[thm]{Remark}
\newtheorem{coro}[thm]{Corollary}
\newtheorem{lem}[thm]{Lemma}
\newtheorem{exm}[thm]{Example}

\makeatletter
\let\c@equation\c@thm
\makeatother
\numberwithin{equation}{section}

\title{Contact resolutions of singular Jacobi structures}

\author{Hichem Lassoued$^{1}$, Camille Laurent-Gengoux$^{2}$}

\address{\emph{1}, Nantes Université, 44000 Nantes, France.}

\address{\emph{2}, Université de Lorraine, CNRS, IECL, F-57000 Metz, France  }

\date{June 11th, 2023}

\begin{document}

\begin{abstract}
  We study contact resolutions of Jacobi structures which are contact on an open subset. We give several classes of examples, as well as classes for which it cannot exist.  
\end{abstract}

\maketitle

\tableofcontents

\section{Introduction} 

Jacobi manifolds are to Poisson manifolds what contact manifolds are to symplectic manifolds. In particular, non-degenerate Jacobi manifolds are exactly contact manifolds.

Their rich local structure makes the study of Jacobi manifolds difficult: for example, there is a singular foliation induced by a Jacobi structure, whose leaves can be either contact manifolds or conformally symplectic manifolds. Thanks to two operations called "Poissonnification" or "Symplectification" \cite{DCW}, many works on Jacobi or contact manifolds simply consist in reducing to works relating to symplectic or homogeneous Poisson manifolds.

  In 2019, \cite{articl} introduced the notion of symplectic resolution of a Poisson structure (symplectic on a dense open set). We have shown that the Poisson structure $xd_xd_y$ on $\Rr^2$ or $\Cc^2$ does not admit a proper symplectic resolution. We have deduced that such a symplectic resolution can exist only in the case where the singular locus of the Poisson structure does not contain any manifold of codimension $1$.
We have also constructed in  \cite{c1234} a symplectic resolution of the Poisson structure $(x^2+y^2)d_xd_y$ on $\Rr^2$, whose singular locus has codimension $2$.

\noindent Here, we extend this work to the Jacobi-Contact framework.

\noindent The notion of contact realization was introduced by Claude Albert \cite{slvj}: a \emph{contact realization} of a Jacobi manifold $(M,\pi_M,E_M)$ and a triple $(\Sigma,\alpha,\varphi)$ where $(\Sigma,\alpha)$ is a contact manifold and $\varphi: \Sigma \rightarrow M$ a surjective submersion on $(M,\pi_M,E_M)$, which is also assumed to be a Jacobi morphism.

It is always possible to find a contact realization of dimension $2n+1$ of a Jacobi manifold of dimension $n$ see \cite{slvj}. It is obviously impossible to find a contact realization of dimension $2n+1$ of a Jacobi manifold of dimension $2n+1$, unless the latter is a contact manifold. But it is possible to weaken the notion of realization by only imposing that $\varphi$ be surjective. More exactly, one defines contact resolutions as follows:
\begin{defi}
A contact resolution of a Jacobi manifold $(M,\pi_M,E_M)$ of dimension $n$ is a triple $(\Sigma,\alpha,\varphi)$ where $(\Sigma,\alpha)$ is a contact manifold of dimension $n$, and $\varphi$ is a surjective Jacobi morphism.
\end{defi}

\noindent A contact resolution is said to be \emph{proper} when $\varphi$ is a proper map, and the resolution is said to be \emph{semi-connected} if no connected component of $\Sigma$ by $\varphi$ has its image included in the singular locus of the Jacobi structure $(\pi_M,E_M)$.
  We will see that, if we suppose $\Sigma$ separable, only Jacobi manifolds which are of contact on a dense open set can admit a contact resolution are the Jacobi manifolds.

\noindent We will give in the first section examples of singular Jacobi manifolds which admit contact resolutions, which shows that the theory is not empty for $\pi=(x^4+y^4)\frac {\partial}{\partial x}\wedge\frac{\partial}{\partial y}+x\frac{\partial}{\partial z}\wedge\frac{\partial}{\partial x}-y \frac{\partial}{\partial y}\wedge\frac{\partial}{\partial z}$ and $E=2d_z$ on $\Rr^3$. Then, in the second section, we show that some Jacobi manifolds do not admit contact resolution, although they are contact over a dense open set, which we show to be a necessary condition for their existence.

\noindent The main tools to show these results are the two ways of defining a Poisson structure from a Jacobi structure. There is the already mentioned notion of Poissonification of Jacobi structures, which is a way of changing a Jacobi manifold $(M,\pi_M,E_M)$ into an homogeneous Poisson manifold $(P,\pi_{P}, Z_{P})$ with $P=M\times\Rr$. Another way is, on a Jacobi manifold $(M,\pi_M,E_M)$ such that the vector field $E_M$ is non-degenerate, to define an homogeneous Poisson structure on a submanifold of codimension $1$ transverse to $E_M$.
We can define a contact resolution of a Jacobi manifold from a symplectic resolution of its Poissonification.
Conversely, we can use the second construction to demonstrate the non-existence of contact resolution in certain situations.

\noindent {\textbf{ Conventions:}}
Throughout this article, we denote by $(M,\Pi_M,E_M)$ a Jacobi structure on a smooth, real analytic or complex manifold $M$, where $\pi_M$ is a bivector field and $E_M $ a vector field. As we will not in general consider two Jacobi structures on the same manifold, this notation is not ambiguous.

\noindent It will often be convenient to view a contact structure on a manifold $\Sigma$ of dimension $2n+1$ as a Jacobi structure $(\Pi_{\Sigma},E_{\Sigma})$ where $\ Pi_{\Sigma}$ is a bivector field and $E_\Sigma$ is a vector field such that the exterior product $E_{\Sigma}\wedge\Pi_{\Sigma}^n$ is not zero at any point of $\Sigma$.
For the equivalence with the usual definition, see \cite{ccc3}.

\noindent Consider a Jacobi structure $(\pi_M,E_M)$ which is in contact over a dense open set.
A point where $E_{M}\wedge\Pi_{M}^n$ is non-zero will be said to be \emph{regular}. It will be said \emph{singular} otherwise. We denote by $\mathbf{M}_{sing}$ the set of singular points of the Jacobi structure $\pi_M$ and $E_M$. We must pay attention here to the vocabulary: for $(M,\pi_M,E_M)$ a Jacobi manifold, it is logical to define $\mathbf{M}_{sing}$ as the set of points where the foliation induces is singular.
For the Jacobi manifolds which are of contact on a dense open set, the two notions coincide.
\begin{rema} \normalfont
On a manifold of dimension $3$, we can have a point $m\in M$ such that $\pi_m\neq 0$, $E_m\neq 0$ and $\pi_m\wedge E_m=0$, i.e. is singular for the Jacobi structure.
For a Jacobi structure $(M,\pi_M,E_M)$ in dimension $3$ such that $E_M \notin \pi_M^{\#}(T^*M)$ on a dense open set, singular points are points $m\in M$ where at least one of the two conditions $\pi_{M\mid_m}=0$ and $E_{M\mid_m}\in\pi_{M\mid_m}^ \#(T_m^*M)$ is checked.
\end{rema}

\section{Existence of contact resolutions in dimension 3}

Any Jacobi structure on $\mathbb R^3$ is of the form:
\begin{equation}\label{yostér}\pi_M=f_z\frac{\partial}{\partial x}\wedge\frac{\partial}{\partial y}+f_y\frac{\partial}{\partial z }\wedge\frac{\partial}{\partial x}+f_x\frac{\partial}{\partial y}\wedge\frac{\partial}{\partial z}, \quad E_M=g_x\frac{\partial}{\partial x}+g_y\frac{\partial}{\partial y}+g_z\frac{\partial}{\partial z},\end{equation}
where $x, y, z$ are the canonical coordinates on $M$, and $f_z,f_y,f_x,g_x,g_y,g_z$ are smooth functions on $M=\mathbb R^3$.

\noindent We show that $E_M$ and $\pi_M$ define a Jacobi structure on $M$ if and only if,

\begin{enumerate}
\item $f_x(\frac{\partial f_y}{\partial z}-\frac{\partial f_z}{\partial y}-g_x)+f_y(\frac{\partial f_z}{\partial x}-\ frac{\partial f_x}{\partial z}-g_y)+f_z(\frac{\partial f_x}{\partial y}-\frac{\partial f_y}{\partial x}-g_z)=0$ by using the hypothesis $\left[\pi_M,\pi_M\right]=2E_M\wedge\pi_M$.

\noindent In particular if
$g_x=\frac{\partial f_y}{\partial z}-\frac{\partial f_z}{\partial y}$, $g_y=\frac{\partial f_z}{\partial x}-\frac{\partial f_x}{\partial z}$ and $g_z=\frac{\partial f_x}{\partial y}-\frac{\partial f_y}{\partial x}$, that is, when $2 E_M=-\dv(\pi_M)$, then the first condition is satisfied.
\item[ ] Note that we can verify that for any bivector field $\pi_M$ on $\mathbb R^3$, we have $${[\pi_M,\pi_M]} =\dv(\pi_M) )\wedge\pi_M$$ and
   $$ [\pi_M,\dv(\pi_M) ] = \dv( \dv(\pi_M)\wedge\pi_M )=-
   \frac{1}{2}\dv([\pi_M,\pi_M] ).$$
   In particular, if $[\pi_M,\pi_M] $ is a $3$-vector that never vanishes, thus defining a volume form, and if the divergence $\dv$ is calculated with respect to this volume form, then $ (\pi_M,\dv(\pi_M))$ is a Jacobi structure.

\item The following relations are satisfied

\[
\left\{
\begin{array}{c @{=} c}
  f_z\left(\frac{\partial g_x}{\partial x}+\frac{\partial g_y}{\partial y}\right)-f_y\frac{\partial g_y}{\partial z}-f_x\frac {\partial g_x}{\partial z}+g_x\frac{\partial f_z}{\partial x}+g_y\frac{\partial f_z}{\partial y}+g_z\frac{\partial f_z}{\partial z}&0\\

f_x\left(\frac{\partial g_z}{\partial z}+\frac{\partial g_y}{\partial y}\right)-f_z\frac{\partial g_z}{\partial x}-f_y\frac {\partial g_y}{\partial x}+g_x\frac{\partial f_x}{\partial x}+g_y\frac{\partial f_x}{\partial y}+g_z\frac{\partial f_x}{\partial z}&0\\

  f_y\left(\frac{\partial g_x}{\partial x}+\frac{\partial g_z}{\partial z}\right)-f_z\frac{\partial g_z}{\partial y}-f_x\frac {\partial g_x}{\partial y}+g_x\frac{\partial f_y}{\partial x}+g_y\frac{\partial f_y}{\partial y}+g_z\frac{\partial f_y}{\partial z}&0
\end{array}
\right.
\]
\noindent by using the assumption $\left[E_M,\pi_M\right]=0$.
\end{enumerate}

\noindent Suppose that the Jacobi manifold $(M,\pi_M,E_M)$ described by (\ref{yostér}) admits a contact resolution $(\Sigma,\alpha,\varphi)$. Since $M$ is of dimension $3$, the map $\varphi$ is of the form $\varphi = (u, v, w)$, where $u$, $v$ and $w$ are smooth functions on $\Sigma$ with values in $\Rr$. For any choice $(p_1,p_2,p_3)$ of local coordinates on $(\Sigma,\alpha)$, the functions $u,v$ and $w$ are functions of the variables $p_1,p_2$ and $p_3 $.

\begin{prop}\label{hmd}
Let $\Sigma$ be a differential manifold of dimension $3$, $\alpha$ a contact structure on $\Sigma$, and $\varphi:\Sigma\rightarrow M$ a smooth surjective map, the triple $(\Sigma, \alpha,\varphi)$ is a contact resolution of the Jacobi manifold $(M,\pi,E)$ described by (\ref{yostér}) if and only if the following six relations hold:

\begin{enumerate}

\item [1.z] \label{1.z}\begin{eqnarray*}
f_z(u,v,w) & =  \{p_1,p_2\}_\Sigma  \left(\frac{\partial u}{\partial p_1}\frac{\partial v}{\partial p_2}-\frac{\partial u}{\partial p_2}\frac{\partial v}{\partial p_1}\right)+
\{p_2,p_3\}_\Sigma  \left(\frac{\partial u}{\partial p_2}\frac{\partial v}{\partial p_3}-\frac{\partial u}{\partial p_3}\frac{\partial v}{\partial p_2}\right)\\ 
& + \{p_3,p_1\}_\Sigma  \left(\frac{\partial u}{\partial p_3}\frac{\partial v}{\partial p_1}-\frac{\partial u}{\partial p_1}\frac{\partial v}{\partial p_3}\right),
\end{eqnarray*}

\item [1.y] \begin{eqnarray*}
f_y(u,v,w) & = \{p_1,p_2\}_\Sigma  \left(\frac{\partial w}{\partial p_1}\frac{\partial u}{\partial p_2}-\frac{\partial w}{\partial p_2}\frac{\partial u}{\partial p_1}\right)+
\{p_2,p_3\}_\Sigma  \left(\frac{\partial w}{\partial p_2}\frac{\partial u}{\partial p_3}-\frac{\partial w}{\partial p_3}\frac{\partial u}{\partial p_2}\right)\\
& + \{p_3,p_1\}_\Sigma  \left(\frac{\partial w}{\partial p_3}\frac{\partial u}{\partial p_1}-\frac{\partial w}{\partial p_1}\frac{\partial u}{\partial p_3}\right),
\end{eqnarray*}

\item [1.x] \begin{eqnarray*}
f_x(u,v,w) & =  \{p_1,p_2\}_\Sigma  \left(\frac{\partial v}{\partial p_1}\frac{\partial w}{\partial p_2}-\frac{\partial v}{\partial p_2}\frac{\partial w}{\partial p_1}\right)+
\{p_2,p_3\}_\Sigma  \left(\frac{\partial v}{\partial p_2}\frac{\partial w}{\partial p_3}-\frac{\partial v}{\partial p_3}\frac{\partial w}{\partial p_2}\right)\\
& +  \{p_3,p_1\}_\Sigma  \left(\frac{\partial v}{\partial p_3}\frac{\partial w}{\partial p_1}-\frac{\partial v}{\partial p_1}\frac{\partial w}{\partial p_3}\right),
\end{eqnarray*}

\item [2.x] $E_{\Sigma}\left[u(p_1,p_2,p_3)\right]=g_x(u,v,w)$, 
\item [2.y] $E_{\Sigma}\left[v(p_1,p_2,p_3)\right]=g_y(u,v,w)$,
\item [2.z] $E_{\Sigma}\left[w(p_1,p_2,p_3)\right]=g_z(u,v,w)$,

\end{enumerate}

\noindent where $\varphi=(u(p_1,p_2,p_3),v(p_1,p_2,p_3),w(p_1,p_2,p_3))$ with $p_1,p_2,p_3$ are local coordinates on $ \Sigma$,
and where $\Pi_{\Sigma}$ and $E_{\Sigma}$ define the Jacobi structure on $\Sigma$ associated with the contact form $\alpha$.

\end{prop}

\begin{proof}[Proof]
This is an immediate calculation. Because $E_{\Sigma},\Pi_{\Sigma}$ can be projected by $\varphi$ onto $E_M,\pi_M$ if and only if the following diagrams commute
\begin{equation}\label{equass}
 \xymatrix{
  C^{\infty}(\Sigma)  \ar[rr]^{E_{\Sigma}}       && C^{\infty}(\Sigma)  \\
  C^{\infty}(M) \ar[rr]^{E_M} \ar[u]_{\varphi^*} && C^{\infty}(M) \ar[u]_{\varphi^*}
  }
\end{equation}
\begin{equation}\label{equass1}
 \xymatrix{
  C^{\infty}(\Sigma)\times C^{\infty}(\Sigma)  \ar[rr]^{\Pi_{\Sigma}}  && C^{\infty}(\Sigma)  \\
  C^{\infty}(M)\times C^{\infty}(M) \ar[rr]^{\pi_M} \ar[u]_{\varphi^*\times\varphi^*}  && C^{\infty}(M) \ar[u]_{\varphi^*}
  }
\end{equation}

\noindent The commutativity of (\ref{equass}) and (\ref{equass1}) is true if and only if it is true for coordinate functions. The relation (\ref{equass}) gives 2.x, 2.y, 2.z and the relation (\ref{equass1}) gives 1.x, 1.y, 1.z.
\end{proof}

\begin{rema}
Proposition \ref{hmd} obviously extends to the case of real analytic and holomorphic Jacobi structures.
\end{rema}
\noindent Let $M$ be a differential manifold of dimension $3$. We consider on a system of local coordinates $(x,y,z)$ in the neighborhood $m\in M$ the field of bivectors $\pi\in\Gamma(\wedge^2TM)$ and a vector field $E \in\Gamma(TM)$ defined by:
\begin{equation}\label{lehbel}
\pi=(x^4+y^4)\frac{\partial}{\partial x}\wedge\frac{\partial}{\partial y}+x\frac{\partial}{\partial z}\wedge\frac{\partial}{\partial x}-y\frac{\partial}{\partial y}\wedge\frac{\partial}{\partial z}, \quad E=\frac{\partial}{\partial z}.
\end{equation}
\noindent A direct calculation shows that this structure is Jacobi. It is obviously of contact on the dense open set $\left\{x \neq 0 \right\}$ or $\left\{y\neq 0\right\}$. The only singular points are the line $\left\{ x=y=0 \right\}$.

\noindent In this case, the Jacobi bracket on $C^{\infty}(M,\Rr)$ is therefore for any pair $(f,g)$ of differentiable functions on $M$ defined by $$\left \{f,g\right\}=i_{\pi}(df\wedge dg)+fi_Edg-gi_Edf.$$

\noindent In order to prove that the contact resolution theory is not empty, we now establish the following result
\begin{prop}\label{refjac} The Jacobi structure described by (\ref{lehbel}) admits contact resolution.
\end{prop}
\begin{proof}[Proof]
Our contact resolution candidate is given by: $\Sigma= \Rr^3$ equipped with the contact structure given by the $1$-form
$$\alpha=dp_3 + \frac{p_1}{\cos^4(p_1^3p_2)+\sin^4(p_1^3p_2)}dp_2 + \frac{3p_2}{\cos^4(p_1^3p_2) +\sin^4(p_1^3p_2)}dp_1,$$ with $(p_1,p_2,p_3)$ the coordinates on $\Sigma= \Rr^3$. Let us verify that this structure is a contact structure. The coefficient "$3$" in the formula giving $\alpha $ was chosen such that $$ d \alpha = \frac{4}{\cos^4(p_1^3 p_2)+\sin^4(p_1 ^3 p_2)} \, \, \, dp_1 \wedge dp_2 .$$
It is therefore obvious that the Reeb field is $\frac{\partial}{\partial p_3}$ and a simple calculation shows that the bivector is given by:
$$  \left(\cos^4(p_1^3 p_2)+\sin^4(p_1^3 p_2)\right)  \frac{\partial}{\partial p_1}   \wedge \frac{\partial}{\partial p_2} 
+
3p_2
\frac{\partial}{\partial p_2}   \wedge \frac{\partial}{\partial p_3} 
+ 
 p_1
\frac{\partial}{\partial p_3}   \wedge \frac{\partial}{\partial p_1} .
   $$
We define the Jacobi morphism $\varphi$ from $(\Sigma,\alpha)$ to $(M,\pi,E)$ by:

\begin{equation}\label{etoill}\begin{array}{rrcl} \varphi :& \Sigma& \to & M \\ & (p_1,p_2,p_3)& \mapsto & (p_1\sin(p_1^3p_2 ),p_1\cos(p_1^3p_2),p_3).
    \end{array}
  \end{equation}
It is easy to verify that this application is surjective. It remains for us to show that this map is a Jacobi morphism, which is why it suffices to verify that the relations of Proposition \ref{hmd} are satisfied. This is done by direct calculation.
In order to simplify the calculations, we use the notations of Proposition \ref{hmd}:

$$u(p_1,p_2,p_3)=p_1\sin(p_1^3p_2),v(p_1,p_2,p_3)=p_1\cos(p_1^3p_2),w(p_1,p_2,p_3)=p_3,$$
and 
$$f_z=x^4+y^4, f_y=x, f_x=-y, g_x=g_y=0, g_z=1.$$ 

\noindent Let us start with the first relation 1.z in Proposition \ref{hmd}.
On the one hand, we have $ f_z(u,v,w) = p_1^4\left(\sin^4(p_1^3p_2)+\cos^4(p_1^3p_2)\right).$
On the other hand, we have
\begin{eqnarray*}
f_z(u,v,w) & = \{p_1,p_2\}_\Sigma  \left(\frac{\partial u}{\partial p_1}\frac{\partial v}{\partial p_2}-\frac{\partial u}{\partial p_2}\frac{\partial v}{\partial p_1}\right)+
\{p_2,p_3\}_\Sigma  \left(\frac{\partial u}{\partial p_2}\frac{\partial v}{\partial p_3}-\frac{\partial u}{\partial p_3}\frac{\partial v}{\partial p_2}\right)\\ 
& + \{p_3,p_1\}_\Sigma  \left(\frac{\partial u}{\partial p_3}\frac{\partial v}{\partial p_1}-\frac{\partial u}{\partial p_1}\frac{\partial v}{\partial p_3}\right)=\{p_1,p_2\}_\Sigma p_1^4+0+0=p_1^4\{p_1,p_2\}_\Sigma.
\end{eqnarray*}
So the first relation is satisfied since $$\{p_1,p_2\}_\Sigma=\left(\sin^4(p_1^3p_2)+\cos^4(p_1^3p_2)\right).$$
  
\noindent Relations 1.y and 1.x give
\begin{eqnarray*}
f_y(u,v,w) &= \{p_1,p_2\}_\Sigma\left(\frac{\partial w}{\partial p_1}\frac{\partial u}{\partial p_2}
-\frac{\partial w}{\partial p_2}\frac{\partial u}{\partial p_1}\right)
+ \{p_2,p_3\}_\Sigma  \left(\frac{\partial w}{\partial p_2}\frac{\partial u}{\partial p_3}-\frac{\partial w}{\partial p_3}\frac{\partial u}{\partial p_2}\right)\\
 &+ \{p_3,p_1\}_\Sigma  \left(\frac{\partial w}{\partial p_3}\frac{\partial u}{\partial p_1}-\frac{\partial w}{\partial p_1}\frac{\partial u}{\partial p_3}\right)
 = 0 - p_1^4\cos(p_1^3p_2)\{p_2,p_3\}_\Sigma \\
 &+ \left(\sin(p_1^3p_2)+3p_1^3p_2\cos(p_1^3p_2) \right) \{p_3,p_1\}_\Sigma=p_1\sin(p_1^3p_2).
\end{eqnarray*}

\noindent and 
 \begin{eqnarray*}
f_x(u,v,w)&= \{p_1,p_2\}_\Sigma \left(\frac{\partial v}{\partial p_1}\frac{\partial w}{\partial p_2}-\frac{\partial v}{\partial p_2}\frac{\partial w}{\partial p_1}\right)+
\{p_2,p_3\}_\Sigma  \left(\frac{\partial v}{\partial p_2}\frac{\partial w}{\partial p_3}-\frac{\partial v}{\partial p_3}\frac{\partial w}{\partial p_2}\right)\\
& + \{p_3,p_1\}_\Sigma  \left(\frac{\partial v}{\partial p_3}\frac{\partial w}{\partial p_1}-\frac{\partial v}{\partial p_1}\frac{\partial w}{\partial p_3}\right) = 0 - p^4\sin(p_1^3p_2)\{p_2,p_3\}_\Sigma \\
& + \left( -\cos(p_1^3p_2)+3p_1^3p_2\sin(p_1^3p_2) \right) \{p_3,p_1\}_\Sigma=-p_1\cos(p_1^3p_2).
\end{eqnarray*}

\noindent These two relations are satisfied if and only if $\{p_3,p_1\}_{\Sigma}=p_1$ and $\{p_2,p_3\}_{\Sigma}=~3p_2$.

\noindent The remaining relations 2.x, 2.y, 2.z give:
  $E_{\Sigma}\left[u(p_1,p_2,p_3)\right]=0=g_1(u,v,w)$,
  $E_{\Sigma}\left[v(p_1,p_2,p_3)\right]=0=g_2(u,v,w)$,
  $E_{\Sigma}\left[w(p_1,p_2,p_3)\right]=1=g_3(u,v,w)$.
\noindent They are therefore satisfied for $E_\Sigma=\frac{\partial}{\partial p_3}$.

\end{proof}

\begin{rema}
\normalfont
The contact resolution given by the theorem \ref{refjac} is connected and therefore semi-connected, but it is not proper.
\end{rema}

\section{Contact Resolution and Poissonification}
We will discuss the link between the contact resolutions and the poissonification. 

\begin{defi}
We call homogeneous Poisson manifold and we denote $(M,\pi_M,Z_M)$, a Poisson manifold $(M, \pi_M)$ equipped with a complete vector field $Z_M$, called homothety field \footnote{All the vector fields that we are going to use in this article are integrable, i.e. they will be vector fields whose flows are defined for all time.}, which verifies
$$\left[Z_M,\pi_M\right]=L_{Z_M}\pi_M=\pi_M.$$

\end{defi}

There is a way to associate an homogeneous Poisson structure with any contact structure.
This is recalled in appendix \ref{appen}.
For the notion of poissonification production, see Proposition \ref{Poissoni}.
\begin{thm}\label{équival}
A Jacobi manifold $(M,\pi_M, E_M)$ of dimension $2n+1$ admits a contact resolution if and only if the poissonification $(P, \pi_P, Z_P)$ admits an homogeneous symplectic resolution.
\end{thm}

\begin{proof}[Proof]
We start with the simplest implication: if there is a contact resolution $(\Sigma,\pi_{\Sigma},E_{\Sigma},\varphi)$ of $(M,\pi_M, E_M) $, then the poissonification manifold $(P, \pi_P, Z_P)$ defined as in Proposition \ref{Poissoni} admits an homogeneous symplectic resolution.
\noindent Consider $(\Sigma,\pi_{\Sigma},E_{\Sigma},\varphi)$ the contact resolution of $(M,\pi_M, E_M)$. By Proposition \ref{symplec}, the poissonnification of the contact manifold $(\Sigma,\pi_{\Sigma},E_{\Sigma})$ is an homogeneous symplectic manifold $(\widetilde{\Sigma}=\Sigma\times\Rr, \pi_{\widetilde{\Sigma}}, Z_{\widetilde{\Sigma}})$. Now let us look at the map $\widetilde{\varphi}=\varphi\times id_{\Rr}$ defined by
\begin{equation}\label{etoill*}\begin{array}{rrcl} \widetilde{\varphi} :&  \widetilde{\Sigma}& \to & P \\ & (\sigma,t)& \mapsto& (\varphi(\sigma),t). 
   \end{array}
 \end{equation}

\noindent It is clear that $Z_{\widetilde{\Sigma}}=\frac{\partial}{\partial t}$ is projectable by $\widetilde{\varphi}$ and projects onto $Z_{P}= \frac{\partial}{\partial t}$. As $\varphi$ is a Jacobi morphism, this implies that $\pi_{\Sigma}$ and $E_{\Sigma}$ can be projected by $\varphi$ onto $M$ and have for projection $\pi_M$ and $E_M$. Consequently, the bivector $\pi_{\Sigma}+Z_{\widetilde{\Sigma}}\wedge E_{\Sigma}$ is projectable by $\widetilde{\varphi}$ onto $P$ and has for projection the bivector $\pi_M +Z_P\wedge E_M$. It follows that the Poisson bivector on $\widetilde{\Sigma}$ associated with the sympectic structure is projectable by $\widetilde{\varphi}$ on $P$ and has for projection the Poisson bivector on $P $. So $\widetilde{\varphi}$ is a Poisson morphism. It is surjective because $\varphi$ is surjective. This implies that $(\widetilde{\Sigma},\Pi_{\widetilde{\Sigma}})$ is a symplectic resolution of $(P,\pi_P)$. As $Z_{\widetilde{\Sigma}}$ projects onto $Z_{P}$ this resolution is homogeneous, this proves the first implication.

\noindent Let us now prove the second implication: if there exists an homogeneous symplectic resolution of
$(P, \pi_P, Z_P)$, then the Jacobi manifold $(M,\pi_M, E_M)$
defined as in Proposition \ref{nown} admits contact resolution.

\noindent First, we consider the subvariety $M\subset P$ given by $\{t=0\}$ in $P$.
Let $(\widetilde{\Sigma}, \pi_{\widetilde{\Sigma}}, Z_{\widetilde{\Sigma}},\widetilde{\varphi} )$ be an homogeneous symplectic resolution of $(P, \pi_P , Z_P)$. Let $\Sigma=\{\sigma\in \widetilde{\Sigma} \mid \ t(\widetilde{\varphi}(\sigma))=0\}$. Since $Z_P[t]=1$, then $Z_{\widetilde{\Sigma}}[\widetilde{\varphi}^*t]=1$, this implies that the function $\widetilde{\varphi}^*t $ has a nonzero differential and therefore $\Sigma\subset \widetilde{\Sigma}$ is a submanifold of $\widetilde{\Sigma}$ transverse to $Z_{\widetilde{\Sigma}}$.
Under Proposition \ref{nown},
there is an induced contact structure $(\Pi_{\Sigma},E_{\Sigma})$ on $\Sigma\subset\widetilde{\Sigma}$. The restriction of $\widetilde{\varphi}$ to $\Sigma$ is surjective on $M$ because $\widetilde{\varphi}$ is surjective. We still have to prove that the restriction is a Jacobi morphism.

\noindent Since $\widetilde{\varphi}$ is a Poisson morphism, then $\Pi_{\widetilde{\Sigma}}$ and $Z_{\widetilde{\Sigma}}$ are projectable by $\widetilde{ \varphi}$ on $P$ and have for projection $\pi_P$ and $Z_P$ respectively. Moreover, the Poisson structures on $\widetilde{\Sigma}$ and $P$ are of the form:

\begin{equation}
\Pi_{\widetilde{\Sigma}}=\pi_{\Sigma}+Z_{\widetilde{\Sigma}}\wedge E_{\Sigma},
\quad
\pi_P=\pi_M+Z_P\wedge E_M,
\end{equation}
\noindent This is only possible if for all $\sigma\in\widetilde{\Sigma}$, we have 
$\wedge^2T_{\sigma}\widetilde{\varphi}(\Pi_{\widetilde{\Sigma}})=\wedge^2T_{\sigma}\widetilde{\varphi}(\pi_{\Sigma})+T_{\sigma}\widetilde{\varphi}(Z_{\widetilde{\Sigma}})\wedge T_{\sigma}\widetilde{\varphi}(E_{\Sigma})=\wedge^2T_{\sigma}\widetilde{\varphi}(\pi_{\Sigma})+Z_P\wedge T_{\sigma}\widetilde{\varphi}(E_{\Sigma})=\pi_M+Z_P\wedge E_M=\pi_P.$

\noindent Hence, $\wedge^2T\widetilde{\varphi}(\pi_{\Sigma})=\pi_M$ and $T\widetilde{\varphi}(E_{\Sigma})=E_M.$
  This implies that $\widetilde{\varphi}\mid_{\Sigma}$ is a Jacobi morphism. This proves the second implication and completes the proof.

\end{proof}

\noindent The above result can be specified as follows
\begin{prop}\label{semiprop}
A smooth (resp, real analytic or holomorphic) Jacobi variety $(M,\pi_M, E_M)$ of dimension $2n+1$ admits a proper contact resolution (resp, semi-connected, resp. separable) if and only if the poissonification $(P,\pi_P,Z_P)$ admits a proper homogeneous symplectic resolution (resp, semi-connected resp. separable).
\end{prop}

\begin{proof}[Proof]
Let $(\Sigma,\alpha,\varphi)$ be a contact resolution of the Jacobi manifold $(M,\pi_M, E_M)$ and $(\widetilde{\Sigma},\Pi_{\widetilde{\Sigma }},Z_{\widetilde{\Sigma}},\widetilde{\varphi})$ an homogeneous symplectic resolution of the poissonification $(P,\pi_P, Z_P)$.
\noindent We are going to show that the contact resolution is proper (resp, semi-connected) if and only if the homogeneous symplectic resolution is proper (resp, semi-connected, resp. separable).

\noindent If $\varphi:\Sigma \to M$ is proper (resp, semi-connected), then $\varphi\times id_{\Rr}:\Sigma\times\Rr \to M\times\Rr$ is proper (resp, semi-connected by the lemma \ref{mel}).

\noindent Conversely, if $\widetilde{\varphi}:\widetilde{\Sigma} \to P$ is proper (resp, semi-connected), then its restriction $\widetilde{\varphi}\mid_{\Sigma}$ on $\Sigma$ is proper because the restriction of a proper function to a closed one is proper. Similarly, if $\widetilde{\varphi}$ is semi-connected, then its restriction $\widetilde{\varphi}\mid_{\Sigma}$ is, because $\Sigma$ is a closed set of $\widetilde {\Sigma}$ ($\Sigma$ is defined by the zeros of the continuous function $\varphi^*t$), it is also semi-connected because $Z_P$ is integrable and any point of $\widetilde{\Sigma}$ can be obtained from a point of $\Sigma$ by following the flow of $Z_P$. So $\widetilde{\Sigma}$ and $\Sigma$ have the same connected components. By Lemma \ref{mel}, a connected component of $\widetilde{\Sigma}$ is sent to the singular locus of $\pi_P$ if and only if its restriction to $\Sigma$ is sent to the locus singular of $(\pi_M,E_M)$. This ends the proof of Proposition.
\end{proof}

From now on, all manifolds will be assumed to be separable, and we place ourselves only in the smooth or real analytic case.
We are going to use \cite{articl} again to show that certain classes of Jacobi structures, although of contact on a dense open set, cannot admit contact resolutions.

\section{Non-existence of contact resolutions}\label{secjac}

The theorem \ref{équival} and Proposition \ref{semiprop} now allow to use the theorems obtained in \cite{articl} in the Poisson case to prove the non-existence of contact resolutions for some classes of Jacobi structures.

\begin{coro}
  If a connected Jacobi manifold admits separable contact resolution, then it is contact over a dense open set. In particular, it is of odd dimension.
\end{coro}
\begin{proof}
According to \cite{articl} only symplectic Poisson manifolds on a dense open set admit a separable symplectic resolution, so the theorem \ref{équival} directly implies the result.
\end{proof}

\noindent The main result in this section is the following.

\begin{thm}\label{cont}

Jacobi manifolds which are contact on a dense open set and whose singular locus contains a submanifold of codimension $1$ do not admit proper contact resolution.
\end{thm}
\begin{proof}[Proof]

Let $(\Sigma,\pi_{\Sigma},E_{\Sigma})$ be a proper contact resolution of the Jacobi manifold $(M,\pi_M,E_M)$ whose singular locus contains a submanifold of codimension $1$. Let $(P,\pi_P,Z_P)$ be the homogeneous Poisson manifold associated with $M$ defined as in Proposition \ref{Poissoni} (poissonification). On the one hand, the Poisson manifold $(P,\pi_P,Z_P)$ admits a proper symplectic resolution according to Proposition \ref{semiprop}. On the other hand, by Proposition \ref{Poissoni}, the Poisson bivector on $P$ is defined by $$\pi_P=\frac{1}{h}\left(\pi_M+Z\wedge E_M \right),$$
where $h :P\rightarrow \Rr$ the homogeneous function of degree $1$ relatively to $Z_P$ defined by
$$ h(t,x)=\exp(t), \quad (t,x)\in P=\Rr\times M.$$

\noindent The bivector $\pi_P$ degenerates on $\mathbf{M}_{sing}\times\Rr$ by the Lemma \ref{mel}. So if the singular locus of $\pi_M$ contains a submanifold of codimension $1$, that of $\pi_P$ also, which contradicts according to \cite{articl} the hypothesis on the existence of a proper symplectic resolution, which completes the proof.

\end{proof}

\begin{thm}\label{4}
Real analytic Jacobi manifolds which are contact over a dense open set and whose singular locus contains a submanifold of codimension $1$ do not admit semi-connected contact resolution.
\end{thm}

\begin{proof}[Proof]

Let $(\Sigma,\pi_{\Sigma},E_{\Sigma})$ be a semi-connected contact resolution of the real analytic Jacobi manifold (note that this remains true in holomorph) $(M,\pi_M ,E_M)$ whose singular locus contains a submanifold of codimension $1$. Let $(P,\pi_P,Z_P)$ be the homogeneous Poisson manifold associated with $M$ defined as in Proposition \ref{Poissoni} (poissonification). On the one hand, the Poisson manifold $(P,\pi_P,Z_P)$ admits a semi-connected symplectic resolution according to Proposition \ref{semiprop}. On the other hand, by Proposition \ref{Poissoni}, the Poisson bivector on $P$ is defined by $$\pi_P=\frac{1}{h}\left(\pi_M+Z\wedge E_M \right),$$
where $h :P\rightarrow \Rr$ the homogeneous function of degree $1$ relatively to $Z_P$ defined by
$$ h(t,x)=\exp(t), \quad (t,x)\in P=\Rr\times M.$$

\noindent The bivector $\pi_P$ degenerates on $\mathbf{M}_{sing}\times\Rr$ by the Lemma \ref{mel}. So if the singular locus of $\pi_M$ contains a submanifold of codimension $1$, that of $\pi_P$ also, which according to \cite{articl} contradicts the hypothesis on the existence of symplectic semi-connected resolution, which completes the proof.

\end{proof}

\noindent In the holomorphic case, the result is much stronger because the submanifold of codimension $1$ which appears in both theorems \ref{cont} and \ref{4} (which apply by changing what is must) still exists and is still of codimension $1$.

\begin{thm}\label{5*}
Except for contact manifolds, holomorphic Jacobi manifolds do not admit semi-connected contact resolutions.
\end{thm}

\begin{proof}[Proof]
The Jacobi structure $(\pi_M,E_M)$ of a non-contact holomorphic Jacobi manifold $M$ of dimension $2n+1$ degenerates to a singular point $m\in M$ if the product exterior $\pi_M^n\wedge E_M$ vanishes at this point. But if it vanishes at at least one point, i.e. if $M$ is not a contact manifold, then the exterior product $\pi_M^n\wedge E_M$ vanishes along a submanifold of codimension $1$ (by Weierstrass preparation theorem which says that the singular locus of any holomorphic function contains regular points, around which the singular locus is a submanifold of codimension $1$). Theorem \ref{4} allows to conclude that no semi-connected contact resolution exists.
\end{proof}

In the smooth case, there are also conditions on the vector field $E$ which make it possible to affirm that certain contact structures do not admit proper or semi-connected resolution.

\begin{prop}
A Jacobi manifold $ (M,\pi,E)$ of contact on a dense open set such that $E$ is an Euler type vector field in the neighborhood of a point cannot admit proper or semi-connected contact resolution.
\end{prop}
\begin{proof}
As $(\pi,E) $ is contact on a dense open set, the dimension of the manifold is an odd integer.
Consider local coordinates in which
  $ E = \sum_{i=1}^{2n+1} x_i \frac{\partial}{\partial x_i} $. We have
   $$ \pi^n \wedge E = P\, \Lambda$$
   where $\Lambda = \frac{\partial}{\partial x_1} \wedge \cdots \wedge \frac{\partial}{\partial x_{2n+1}} $ and where $ P $ is a function. As $L_E \Lambda = -(2n+1) \Lambda $, we have:
    $$ 0 = L_E (\pi^n \wedge E) = \left(E[P] -(2n+1)\, P\right) \Lambda $$
This implies that $E[P] =(2n+1)P$, in other words, that $P$ is an homogeneous polynomial of degree $ 2n+1$. For an homogeneous polynomial of odd degree on $\mathbb R $, $P=0 $ contains a submanifold of codimension $1$. This submanifold being exactly the singular locus of $ (E,\pi)$, Theorem \ref{cont} allows us to conclude that there is no proper contact resolution.
\end{proof}

\begin{rema}
\normalfont
The proof above shows that if proper resolution exists, then the vector field $E$ cannot be
of the Euler type with weight: $\sum_{i=1}^{2n+1} \omega_i x_i \frac{\partial}{\partial x_i} $, with $\omega_i \in \mathbb N^*, $
only if the sum of the weights $\sum_{i=1}^{2n+1} \omega_i $ is a positive even integer.
  If it is an odd integer, then the function is polynomial and changes sign, which implies that there is a submanifold of codimension $1$ where it is zero.
\end{rema}

In the remainder of this section we introduce examples of bivector fields and vector fields on $\Rr^n$ which we will prove that they define a Jacobi structure if we impose additional conditions on them. Then we describe large classes of Jacobi manifolds which cannot admit proper contact resolutions. Some of the statements in this section relate to real analytic and holomorphic Jacobi manifolds, some relate to the smooth case. The context will be clearly indicated if required.
\begin{prop}\label{pipi1}
We consider on $M=\Rr^{2m+1}$ equipped with coordinates $x,z,y_1,\ldots,y_{2m-1}$, the bivector field defined by $$\pi_M=\left(\sum_ {i=1}^{2m-1}f(y_i)\frac{\partial}{\partial y_i}\right)\wedge\frac{\partial}{\partial z}-\left(\sum_{i =1}^{2m-1}y_i^nf(y_i)\frac{\partial}{\partial y_i}\right)\wedge\frac{\partial}{\partial x},$$
and the vector field defined by $$E_M=\left(\sum_{i=1}^{2m-1}g(y_i)\right)\frac{\partial}{\partial x}+\left(\ sum_{i=1}^{2m-1}h(y_i)\right)\frac{\partial}{\partial z}.$$
where $n$ is a positive integer and $f,g,h$ are functions on $\Rr$. For any choice of $f$, there exists $g$ and $h$ such that the bivector $\pi_M$ and the vector field $E_M$ define a Jacobi structure on $\Rr^{2m+1}$ if $n \geqslant 2$ and the function $f$ is polynomial without constant term or if $n=1$ and $f$ is any polynomial.
\end{prop}

\begin{proof}[Proof]
A direct calculation shows that
\begin{equation*}
\begin{split}
\left[\pi_M,\pi_M\right] & = 2\sum_{i=1}^{2m-1}\left(\left(\left\{x,\left\{y_i,z\right\}\right\}+\left\{y_i,\left\{z,x\right\}\right\}+\left\{z,\left\{x,y_i\right\}\right\}\right)\frac{\partial}{\partial x}\wedge\frac{\partial}{\partial y_i}\wedge\frac{\partial}{\partial z}\right)\\
 & = -2n\left(\sum_{i=1}^{2m-1}y_i^{n-1}f^2(y_i)\frac{\partial}{\partial x}\wedge\frac{\partial}{\partial y_i}\wedge\frac{\partial}{\partial z}\right),
\end{split}
\end{equation*}
 and
 \begin{equation*}
  \left[\pi_M,\pi_M\right] =  2E_M\wedge\pi_M=-2n\left(\sum_{i=1}^{2m-1}y_i^{n-1}f^2(y_i)\frac{\partial}{\partial x}\wedge\frac{\partial}{\partial y_i}\wedge\frac{\partial}{\partial z}\right).
  \end{equation*}

Also, $\left[E_M,\pi_M\right]=\left(g'(y_i)f(y_i)+y_i^nh'(y_i)\right)\frac{\partial}{\partial x}\wedge\frac{\partial}{\partial z}=0$.
The two identities $[\pi_M,\pi_M] = 2E_M\wedge\pi_M$ and $[E_M,\pi_M] = 0$ impose that the functions $g$ and $h$ must satisfy the relations for all $i=1 ,...,m$,

\begin{equation}
\label{jojo}
\left \{
\begin{array}{c @{=} c} 
   g(y_i)f(y_i) + f(y_i)h(y_i)y_i^n & -ny_i^{n-1}f^2(y_i) \\
   g'(y_i)f(y_i) + y_i^n h'(y_i)f(y_i) & 0. \\
\end{array}
\right.
\end{equation}
Which is equivalent to the linear system,
\begin{equation}
\label{jojo1}
\left \{
\begin{array}{c @{=} c} 
   g(t) + h(t)t^n & -nt^{n-1}f(t) \\
   g'(t) + t^n h'(t) & 0 .\\
\end{array}
\right.
\end{equation}
\noindent By linearity of (\ref{jojo1}), it suffices to show the existence of a solution for $f$ a monomial function.

\noindent Let us take the case $n=1$, we consider $f(t)=t^k$, $g(t)=at^k$ and $h(t)=bt^{k-1}$. Replace the functions $g,h$ and $f$ in (\ref{jojo1}) by these polynomials gives
\begin{equation*}
\left \{
\begin{array}{c @{=} c} 
   a + b & -1 \\
   ka + (k-1)b & 0 \\
\end{array}
\right.
\end{equation*}
\noindent It is easy to see that this system admits at least one solution for any integer $k \geqslant 0$. If $k=0$ we have $b=0$ and $h(t)=0$, so there is no term proportional to $\frac{1}{t}$. This implies that the system (\ref{jojo1}) admits a polynomial solution.

\noindent Concerning the case $n \geqslant 2$, we will prove the existence of solution for any polynomial $f$ without the constant term. By linearity of (\ref{jojo1}), it suffices to show the existence of a solution for $f$ a monomial. We consider $f(t)=t^{k}$, $g(t)=at^{k+n-1} $ and $ h(t)=bt^{k-1}$ for any integer $ k \geqslant n$. Replacing the functions $g,h$ and $f$ in (\ref{jojo1}) by these polynomials gives,

\begin{equation*}
\left \{
\begin{array}{c @{=} c} 
   a + b & -n \\
   (k+n-1)a + (k-1)b & 0 \\
\end{array}
\right.
\end{equation*}

\noindent This system of equations always admits a unique solution. Since $k\geqslant 1$, $h(t)$ is indeed a polynomial and therefore for any polynomial $f$ without the constant term the system (\ref{jojo1}) admits a solution.

\noindent If we assume that $f$ is a constant function, it is easy to see that the system does not admit a solution for all functions $g,h$.
\end{proof}
\begin{rema}\normalfont
If the function $f$ is a polynomial of degree $1$, then the functions $g$ and $h$ are constant functions which makes the vector field $E_M$ non-degenerate at any point of the considered manifold $M$ .
\end{rema}

\noindent To make things clearer, we give a series of examples.
\begin{exm}\normalfont \label{1}
Equip $\Rr^3$ with the bivector field $\pi$ defined by
$$\pi=(2+3y)\frac{\partial}{\partial y}\wedge\frac{\partial}{\partial z}+y(2+3y)\frac{\partial}{\partial x}\wedge\frac{\partial}{\partial y},$$
and the vector field $E$ defined by
$$E=-2\frac{\partial}{\partial x}-3\frac{\partial}{\partial z}.$$
In this example, we have $f(y)=2+3y$ so a direct calculation shows that $$\left[\pi,\pi\right]=-2f^2\frac{\partial}{\partial x }\wedge\frac{\partial}{\partial y}\wedge\frac{\partial}{\partial z}=2(-9y^2-4-12y)\frac{\partial}{\partial x} \wedge\frac{\partial}{\partial y}\wedge\frac{\partial}{\partial z}, \quad \left[E,\pi\right]=0.$$
\end{exm}
\begin{exm}\normalfont \label{2}
Equip $\Rr^3$ with the bivector field $\pi$ defined by
$$\pi=(y^3+y^2+y)\frac{\partial}{\partial y}\wedge\frac{\partial}{\partial z}+y(y^3+y^2 +y)\frac{\partial}{\partial x}\wedge\frac{\partial}{\partial y},$$
and the vector field $E$ defined by
$$E=(2y^3+y^2)\frac{\partial}{\partial x}+(-3y^2-2y-1)\frac{\partial}{\partial z}.$$
In this example, we have $f(y)=y^3+y^2+y$ so a direct calculation shows that $$\left[\pi,\pi\right]=-2f^2\frac{\partial}{\partial x}\wedge\frac{\partial}{\partial y}\wedge\frac{\partial}{\partial z}, \quad \left[E,\pi\right]=0.$$
\end{exm}

\begin{exm}\normalfont\label{3}
Equip $\Rr^3$ with the field of bivectors $\pi$ defined by
$$\pi=(y^2+y+1)\frac{\partial}{\partial y}\wedge\frac{\partial}{\partial z}+y(y^2+y+1)\ frac{\partial}{\partial x}\wedge\frac{\partial}{\partial y},$$
and the vector field $E$ defined by
$$E=(y^2-1)\frac{\partial}{\partial x}+(-2y-1)\frac{\partial}{\partial z}.$$
In this example, we have $f(y)=y^2+y+1$ so a direct calculation shows that $$\left[\pi,\pi\right]=-2f^2\frac{\partial} {\partial x}\wedge\frac{\partial}{\partial y}\wedge\frac{\partial}{\partial z}, \quad \left[E,\pi\right]=0.$$
\end{exm}

\begin{rema}
\normalfont
When $f$ is a constant function in the previous examples and $n=1$, then we obtain a contact structure.
\end{rema}
\begin{prop}
The Jacobi manifolds of Examples \ref{1},\ref{2} and \ref{3} contain a submanifold of codimension $1$ in their singular loci if and only if the polynomial $f$ admits a zero. 
\end{prop}
\noindent The proof is obvious. The submanifolds $\{y=y_0\}$ with $y_0$ being a zero of $f$ is contained in the singular locus.

\noindent
\begin{coro}
The Jacobi manifolds of Proposition \ref{pipi1} have no proper contact resolution (resp, semi-connected) unless $n=1$ and $f$ is constant, in which case they are themselves contact.
\end{coro}

\section{Appendix.}\label{appen}

\subsection{Homogeneous structures and poissonification}

\noindent We provide the definition of "homogeneous symplectic resolution" that we need throughout this article.

\begin{defi} Let $(M,\pi_M,Z_M)$ be an homogeneous Poisson manifold.
An homogeneous symplectic resolution of $(M,\pi_M,Z_M)$ is a symplectic resolution $(P,\pi_P,\varphi)$ such that $P$ is homogeneous via a vector field $Z_P$ on $P$ which can be projected by $\varphi$ onto $M$ and has a projection of $Z_M$.
\end{defi}

\noindent We give here a classical result (see \cite{slvj1}, for more details) useful for this article, and which allows us to construct an homogeneous Poisson manifold from a Jacobi manifold.

\begin{prop}\label{Poissoni}
Let $(M,\pi_M,E_M)$ be a Jacobi manifold. We set $$P=\Rr\times M.$$
We note $t$ the canonical coordinate on the factor $\Rr$, and $Z = \frac{\partial}{\partial t}$. Let $h :P\rightarrow \Rr$ be the homogeneous function of degree $1$ relatively to $Z$ defined by
$$ h(t,x)=\exp(t),\ \text{for all} \ (t,x)\in P=\Rr\times M.$$
We define on $P$ a bivector $\pi_P$ by setting $$\pi_P=\frac{1}{h}\left(\pi_M+Z\wedge E_M \right).$$
The triple $(P,\pi_P,Z_P)$ is an homogeneous Poisson manifold.
\end{prop}
\begin{proof}[Proof]
The proof consists in demonstrating that $\left[\pi_P,\pi_P\right]=0$, this is done by a direct computation using the Jacobi identities: $\left[\pi_M,\pi_M\right]=2E_M \wedge\pi_M$ and $\left[E_M,\pi_M\right]=~0.$
\end{proof}
\noindent
\begin{defi}
We call the poissonification of the Jacobi manifold $(M,\pi_M,E_M)$ the homogeneous Poisson manifold $(P,\pi_P,Z_P)$ constructed in Proposition \ref{Poissoni}
\end{defi}

\begin{lem}\label{mel}
Let $\mathbf{M}_{sing}$ be the singular locus of a Jacobi manifold $(M,\pi_M,E_M)$. The singular locus of the poissonification $(P,\pi_P,Z_P)$ is $\mathbf{M}_{sing}\times\Rr$.
\end{lem}
\begin{proof}[Proof]
Let us return to the context of the previous Lemma. According to the definition of the Poisson structure $\pi_P$ on the poissonification $(P,\pi_P,Z_P)$, the latter degenerates when the Jacobi structure $(\pi_M,E_M)$ degenerates. Indeed, we have for any positive integer ~$k$,
$$\pi_P^{k+1}=(k+1)\frac{\pi_M^k\wedge E_M\wedge Z_P}{h^{k+1}}+\frac{\pi_M^{k+1 }}{h^{k+1}}.$$

\noindent For $k$ is the rank of $\pi_M$, we have $\pi_M^{k+1}=0$ and we have $\pi_P^{k+1}=(k+1)\frac{ \pi_M^k\wedge E_M\wedge Z_P}{h^{k+1}}$. This ($2k+2$)-vector field is zero at $m\in M$ if and only if $m$ is a singular point of $\pi$ or if and only if $E_{M\mid_m}\in \Ima\pi_{M\mid_m}^ {\#}$. So $\mathbf{P}_{sing}=\mathbf{M}_{sing}\times\Rr$
\end{proof}
\noindent The following Proposition gives the construction of a symplectic manifold from a contact manifold.

\begin{prop}\label{symplec}

The poissonification of a contact manifold $(M,\pi_M,E)$ is an homogeneous symplectic manifold, denoted by $(P,\pi_P,Z_P)$.
\end{prop}
\noindent In this case we call this operation "to symplectify" a contact manifold.

\begin{proof}[Proof]
The proof of this Proposition consists in proving that the Poisson structure on the poissonification manifold is non-degenerate.
Let $(P,\pi_P,Z_P)$ be an homogeneous Poisson manifold constructed from a Jacobi manifold $(M, \pi_M, E_M)$. By Proposition \ref{Poissoni}, the Poisson structure on $P$ is defined by

$$\pi_P=\frac{1}{h}\left(\pi_M+Z_P\wedge E_M \right),$$
where $Z_P=\frac{\partial}{\partial t}$ is the vector field associated with the canonical coordinate $t$ on
the factor $\Rr$, and $h :P\rightarrow \Rr$ is an homogeneous function of degree $1$ relatively to $Z$ defined by
$$ h(t,x)=\exp(t), \quad (t,x)\in P=\Rr\times M.$$

\noindent A direct calculation gives $\pi_P^{n+1}=\frac{1}{h^n}\pi_M^n\wedge Z_P\wedge E_M$. As $\frac{1}{h^n}\pi_M^n\wedge E_M$ is not zero at any point of $\wedge^{2n}TM$ and $Z_P=\frac{\partial}{\partial t}$ is not zero at any point, then $\pi_P$ is symplectic on $P$. This completes the proof.
 
\end{proof}

\noindent We are going to give here without demonstration a useful proposition to link symplectic resolutions and contact resolutions. For the proof see \cite{slvj1}

\begin{prop}\label{nown}
Let $(M,\pi_M,Z)$ be an homogeneous Poisson manifold, and $N$ a submanifold of codimension $1$ of $M$ transverse to the dilation field $Z$, then there exists on $N$ a Jacobi structure $(\pi_N,E_N)$. This Jacobi structure on $N$ is said to be induced by the homogeneous Poisson structure of $M$.
It verifies $$\pi_M|_n=\pi_N|_n+Z\wedge E_N|_n$$
for any point $n \in N$.
If $(M,\pi_M,Z)$ is symplectic and $Z$ never vanishes, then the induced Jacobi structure is contact.
\end{prop}

\subsection{Non-existence of symplectic resolutions}

We recall here the main result of \cite{articl}, obtained by merging Theorems 2.21, 2.25 and 2.26.

\begin{thm}\cite{articl}
Let $(M,\pi)$ be a symplectic Poisson manifold over a dense open set. Suppose there is a symplectic resolution $(\Sigma, \Pi_\Sigma, \phi_\Sigma)$
\begin{enumerate}
\item In the holomorphic case, the Poisson structure $(M,\pi) $ must itself be a symplectic manifold.
\item In the smooth (real analytic) case, if $\phi_\Sigma $
  is proper (resp, semi-connected), then the singular locus of $\pi $ cannot contain a submanifold of codimension $1$. \end{enumerate}
\end{thm}

\end{document}